\documentclass[reqno]{amsart}

\usepackage{amssymb,amsthm,amsfonts,amstext}
\usepackage{amsmath}
\usepackage{zito}
\usepackage{mathptmx}       
\usepackage{helvet}         
\usepackage{courier}        
\usepackage{pstricks}

\usepackage[notcite,notref]{showkeys}

\makeatletter \@addtoreset{equation}{section} \makeatother

\renewcommand\thefigure{\thesection.\@arabic\c@figure}
\renewcommand\thetable{\thesection.\@arabic\c@table}
\newtheorem{theorem}{Theorem}[section]
\newtheorem{lemma}[theorem]{Lemma}
\newtheorem{proposition}[theorem]{Proposition}

\theoremstyle{definition}

\theoremstyle{remark}

\newcommand{\mc}[1]{{\mathcal #1}}

\newcommand{\bb}[1]{{\mathbb #1}}

\newcommand{\plus}{\!+\!}
\newcommand{\minus}{\!-\!}
\newcommand{\<}{\langle}
\renewcommand{\>}{\rangle}

\DeclareMathOperator{\supp}{supp} 

\DeclareMathOperator{\Geom}{Geom}

\DeclareMathOperator{\Var}{Var}

\title{Equilibrium fluctuations for a discrete Atlas model}
\date{}
\author{F.~Hern\'andez, M.~Jara, F.~Valentim}

\subjclass[2010]{60K35, 60H15, 82C22}
\keywords{Atlas model, zero-range process, equilibrium fluctuations}

\address{\noindent 
Instituto de Matem\'atica, Universidade Federal Fluminense, Rua M\'ario Santos Braga S/N,
Niter\'oi, RJ 24020-140, Brazil
\newline
e-mail: \rm \texttt{freddyhernandez@id.uff.br}}

\address{\noindent IMPA, Estrada Dona Castorina 110, Rio de Janeiro, RJ, 22460-320, Brazil.
\newline
e-mail: \rm \texttt{mjara@impa.br}}

\address{
\noindent
Departamento de Matem\'atica, Universidade Federal do Esp\'irito Santo, Av. Fernando Ferrari, 514, Goiabeiras, Vit\'oria, 29075-910, Brazil.
\newline
e-mail \rm\texttt{fabio.valentim@ufes.br}
}

\begin{document}

\begin{abstract}
We consider a discrete version of the Atlas model, which corresponds to a sequence of zero-range processes on a semi-infinite line, with a source at the origin and a diverging density of particles. We show that the equilibrium fluctuations of this model are governed by a stochastic heat equation with Neumann boundary conditions. As a consequence, we show that the current of particles at the origin converges to a fractional Brownian motion of Hurst exponent $H=\frac{1}{4}$. 
\end{abstract}

\maketitle

\section{Introduction}

The so-called {\em Atlas model} can informally  described as a semi-infinite system of independent Brownian motions on $\bb R$, on which the leftmost particle receives a drift towards the right of strength $\gamma>0$. This model is the simplest example of diffusions with {\em rank-based} interactions. These diffusions interacting through their rank have been proposed as a simple model for  capitalizations in equity markets, see \cite{Fer1}, \cite{IchPapBanKarFer} and the references therein.
In \cite{DemTsa} it is proved that the equilibrium fluctuations of the Atlas model converge in a proper scale to the solution of a stochastic heat equation with reflection at the origin. The ultimate interest of this fluctuation result is that the fluctuations of the leftmost particle can be expressed in the limit as a (singular) linear functional of the solution of the stochastic heat equation. In particular one of the main results of \cite{DemTsa} is that the fluctuations of the leftmost particle are governed by a fractional Brownian motion of Hurst exponent $\frac{1}{4}$.

A key observation in order to derive various results about the Atlas model, is that the sequence of {\em spacings} between the Brownian motions follows a Markovian evolution, for which an invariant measure of product form is known to exist.

A natural question turns out to be what happens with a discrete analogous of the Atlas model. A simple discrete analogous of the Atlas model consists on a semi-infinite system of particles, on which the first one has a drift towards the right and subjected to the so-called {\em exclusion rule}: no particles can share the same position at any time $t$. 
It is well known that such a system of exclusion particles is in bijection with a zero-range process with a reservoir of particles at the origin. We learned this bijection from \cite{LanOllVol}; although the system treated in \cite{LanOllVol} is bi-infinite, the same bijection works for semi-infinite systems of exclusion particles. 
The earliest reference we were able to find on which this bijection appears is \cite{Fer2}. 
In \cite{LanOllVol}, \cite{LanVol} this relation was exploited to obtain various scaling limits of observables of the exclusion process as a consequence of convergence results for analogous quantities in the zero-range process. 

For this reason, we study in this article the {\em stationary current fluctuations} of a zero-range process with a source at the origin. A formal description of this process is the following. Particles live on the semi-infinite lattice $\bb N=\{1,2,\dots\}$. At each site of the lattice $\bb N$ there is a Poissonian clock of rate $2$. Each time the clock at site $x \in \bb N$ rings, one of the particles at site $x$ moves to $x-1$ or $x+1$ with equal probability. If the particle decides to move to $y=0$, then the particle leaves the system. In addition, with exponential rate $\lambda$ a particle is created at site $x=1$. Burke's theorem says that for $\lambda <1$ the product measure with geometric marginals of success rate $\lambda$ is invariant under this evolution. Denote by $J_x(t)$ the signed current of particles between sites $x$ and $x+1$ up to time $t \geq 0$. For $x=0$, $J_0(t)$ denotes the number of particles that entered into the system, minus the number of particles that left the system up to time $t \geq 0$. The current $J_x(t)$ denotes exactly the displacement of the $x$-th particle up to time $t$ in an exclusion process where particles are ordered from left to right. All but the leftmost exclusion particle are symmetric; the leftmost particle jumps to the right with rate $1$ and it jumps to the left with rate $\lambda$. In \cite{LanVol}, it is shown that under a diffusive scaling, the space-time limit of the fluctuations of the density of particles is given by a conservative stochastic heat equation, better known as an {\em Ornstein-Uhlenbeck equation}. With some extra work, this result can be used to derive the scaling limit of the current fluctuations as well.

Let $n \in \bb N$ be a scaling parameter. If we want the exclusion process to serve as a discrete approximation of the Atlas model, it is reasonable to scale $\lambda$ with $n$ in such a way that the leftmost particle behaves like a Brownian motion with drift in the limit $n \to \infty$. Therefore, it is reasonable to take $\lambda_n = 1 -\frac{b}{n}$, where $b>0$ is the drift of the limiting Brownian motion. We will start the zero-range process with the invariant measure associated to $\lambda_n$, namely a product of geometric distributions with parameter $1-\lambda_n$. Notice that the average number of particles per site grows linearly with $n$.

It turns out that the proof of \cite{LanVol} breaks down in that case. The heart of the proof of the convergence of the density fluctuation field is the so-called {\em Boltzmann-Gibbs principle}, which roughly states that any observable of a conservative particle system is asymptotically equivalent to a linear functional of the density of particles. 
The main issue is that the density of particles per site is equal to $\frac{n}{b}-1$ and it grows to infinity as $n \to \infty$. Because of that, a key compactness argument in the classical proof of the so-called Boltzmann-Gibbs principle does not work. In \cite{GonJar1}, a quantitative proof of the Boltzmann-Gibbs principle was proposed. It turns out that in our situation, this alternative proof allows to circumvent the compactness argument by the use of the so-called {\em spectral gap inequality}, which gives a sharp bound on the largest eigenvalue of the dynamics restricted to a finite box. For the version of the zero-range process considered in this article, the spectral gap inequality was proved in \cite{Mor}. For zero-range processes with other interaction rates, see \cite{LanSetVar} and \cite{Nag}. According to \cite{Mor}, the spectral gap of the zero-range process presented here has a non-trivial dependence on the total number of particles. Therefore, the proof of the quantitative Boltzmann-Gibbs principle of \cite{GonJar1} needs to be adapted accordingly. 

With the Boltzmann-Gibbs as principal tool, we prove that in the stationary state the space-time fluctuations of the current converge to the solution of the stochastic heat equation 
\[
\partial_t  X = b^2 \Delta X + \sqrt{2} \dot{\mc W}.
\]
with zero initial datum.
In particular, as in \cite{DemTsa} the current of particles through the origin converges to a Brownian motion of Hurst parameter $\frac{1}{4}$.

This article is organized as follows. In Section \ref{s2} we introduce the zero-range process with a source at the origin, we define what do we understand by the {\em current fluctuation field}  and we state the main results of this article. In Section \ref{s3} we state and give a sketch of proof of various estimates related to the variance of additive functionals of Markov processes. The exposition follows closely the one of \cite{GonJar1}, and proofs already included in \cite{GonJar1} are omitted. In Section \ref{s4} we prove the main results of this article, and in Section \ref{s5} we discuss possible generalizations of the results proved in this article.

\section{Notation and definitions}
\label{s2}
\subsection{The model} 

Let $\Omega_0 = \bb N_0^{\bb N}$\footnote{We denote $\bb N_0 = \{0,1,2,\dots\}$ and $\bb N=\{1,2,\dots\}$.} be the state space of a Markov process which we will describe below. We denote by $\eta = \{\eta(x); x \in \bb N\}$ the elements of $\Omega_0$ and we call them {\em particle configurations}. We call the elements $x \in \bb N$ {\em sites} and we say that $\eta(x)$ is the number of particles at site $x$ on the configuration $\eta$. For $x,y \in \bb N$ and $\eta \in \Omega_0$ such that $\eta(x) \geq 1$ let $\eta^{x,y} \in \Omega_0$ be given by
\[
\left\{
\begin{array}{c@{;\;}l}
\eta(x)-1 & z=x\\
\eta(y)+1 & z=y\\
\eta(z) & z \neq x,y.\\
\end{array}
\right.
\]
In other words, $\eta^{x,y}$ is obtained from $\eta$ by moving a particle from site $x$ to site $y$. For $\eta \in \Omega_0 $ such that $\eta(1) \geq 1$ we define $\eta^{1,0} = \eta - \delta_1$ and for $\eta \in \Omega_0$ we define $\eta^{0,1} = \eta + \delta_1$, where $\delta_1(x) = 0$ if $x \neq 1$ and $\delta_1(x)=1$ if $x =1$. In other words $\eta^{0,1}$ is obtained from $\eta$ by adding a particle at site $x=1$. 

We say that a function $f: \Omega_0 \to \bb R$ is {\em local} if there exists a finite set $A \subseteq \bb N$ such that whenever $\eta(x) = \xi(x)$ for all $x \in A$, $f(\eta) = f(\xi)$. In that case we say that the {\em support} of $f$ is contained in $A$, and we denote it by $\supp(f) \subseteq A$. 

Let $\lambda \in (0,1)$ be given. Let $g: \bb N_0 \to \bb R$ be given by $g(\ell)=1$ if $\ell \neq 0$ and $g(0)=0$. For $f: \Omega_0 \to \bb R$ local we define $L f: \Omega_0 \to \bb R$ as
\[
L f(\eta) = \sum_{x \in \bb N} g(\eta(x)) \big[ f(\eta^{x,x+1})+f(\eta^{x,x-1})-2f(\eta)\big] 
		+ \lambda \big[ f(\eta^{0,1})-f(\eta) \big]
\]
for any $\eta \in \Omega_0$. Although $\eta(0)$ is not defined, it will be convenient to adopt the convention $g(\eta(0))=\lambda$.
Let us explain briefly Andjel's construction of the Markov process associated to the linear operator $L$ defined in this way. First we restrict ourselves to the set
\[
\big\{ \eta \in \Omega_0; \sum_{x \in \bb N} \eta(x) e^{-Mx} < \infty\},
\]
where $M$ is a fixed constant. We will call this set $\Omega$ and we equip it with the product topology. Notice that local functions are indistinctly defined in $\Omega$ or $\Omega_0$. We say that a local function $f: \Omega \to \bb R$ is {\em Lipschitz} if there are $K >0$ and $A \subseteq \bb N$ finite such that
\[
\forall \eta, \xi \in \Omega,  \big| f(\eta) - f(\xi) \big| \leq K \sum_{x \in A} \big| \eta(x) -\xi(x)\big|.
\]
The closure of the operator $L$ restricted to local Lipschitz functions turns out to be the generator of a Markov process $\{\eta_t; t \geq 0\}$ in $\Omega$, which we call the {\em zero-range process} with a source at the origin. The dynamics of this process is not hard to describe. At each site $x \in \bb N$ the process waits an exponential time of rate $2$, at the end of which one particle jumps from $x$ to one of this two neighbors with equal probability. If there are no particles at site $x$ at the moment of jump, nothing happens. At $x=1$, if  the aforementioned particle decides to jump left, it disappears. Additionally, with exponential rate $\lambda$ a particle is created at site $x=1$.

Let $\mu_\lambda$ denote the product geometric measure on $\Omega$:
\[
\mu_\lambda(d\eta) = \prod_{x \in \bb N} (1-\lambda) \lambda^{\eta(x)} d\eta(x).
\]
In principle, $\mu_\lambda$ is defined in $\Omega_0$, but it puts total mass on $\Omega$. It can be verified that $\mu_\lambda$ is invariant and reversible under the evolution of $\{\eta_t; t \geq 0\}$. 

\subsection{The current fluctuations}
\label{s2.2}
Let $n \in \bb N$ be a scaling parameter. Let $b >0$ be fixed and let $\{\lambda_n; n \in \bb N\}$ be defined for simplicity as
$\lambda_n = 1 - \frac{b}{n}$ for any $n \geq b$. All the results on this article can be easily generalized to the case of a sequence $\{\lambda_n; n\in \bb N\}$ in $(0,1)$ such that
\[
\lim_{n \to \infty} n(1-\lambda_n) = b.
\]
Let us denote by $\mu^n$ the measure $\mu_{\lambda_n}$ and
let us denote by $\{\eta_t^n; t \geq 0\}$ the process $\{\eta_{tn^4}; t \geq 0\}$ with initial distribution $\mu^n$. The time scaling may seem a little mysterious right now, but it will turn out to be the right one in our setting. We denote by $\bb P_n$ the distribution of $\{\eta_t^n; t \geq 0\}$ and by $\bb E_n$ the expectation with respect to $\bb P_n$. All these parameters will be fixed from now on and up to the end of this article.

For each $n \in \bb N$ let $J_t^n(x)$ denote the signed number of particles crossing the bond $\{x,x+1\}$ up to time $t$. Similarly, we denote by $J_t^n(0)$ the number of particles created at $x=1$ minus the number of particles annihilated at $x=1$, up to time $t$. Our aim is to obtain the scaling limit of the current process $\{J_t^n(x); t \geq 0, x \in \bb N_0\}$ as $n \to \infty$. For each function $f \in \mc C_c^\infty([0,\infty))$ we define 
\begin{equation}
\label{def}
X_t^n(f) = \frac{1}{n^{5/2}} \sum_{x \in \bb N_0} J_t^n(x) f \big( \tfrac{x}{n}\big) 
+ \frac{1}{n^{3/2}} \sum_{x \in \bb N} \big( \eta_0^n(x) -\rho_n\big) F\big(\tfrac{x}{n}\big)
\end{equation}
for any $t \geq 0$ and any $n \in \bb N$. Here $F: [0,\infty) \to \bb R$ is defined as $F(x) = -\int_x^{\infty} f(y) dy$ for any $x \geq 0$ and $\rho_n = \frac{n}{b}-1 = \int \eta(x) d\mu^n$ is the expected number of particles per site.

In this way we have defined a measure-valued process $\{X_t^n; t \geq 0\}$ which we will call the {\em current fluctuation field} associated to the zero-range process $\{\eta_t^n; t \geq 0\}$. The extra term involving $F$ and $\eta_0^n$ may seem strange right now, but it will allow to get rid of a static drift term on the scaling limit of the current fluctuation field.
%
%

\subsection{The scaling limit}

Let $\dot{\mc W}$ be a standard space-time white noise on $[0,\infty) \times [0,\infty)$. We say that a measure-valued process $\{X_t; t \geq 0\}$ is a {\em martingale solution} of the stochastic heat equation
\begin{equation}
\label{SHE}
\partial_t X = b^2\Delta X + \sqrt{2} \dot{\mc W} 
\end{equation}
with reflecting boundary conditions at $x=0$ if for any function $f \in \mc C_c^\infty([0,\infty))$ such that $f'(0)=0$, the process
\[
X_t(f) - b^2\int_0^t X_s(\Delta f) ds
\]
is a continuous martingale of quadratic variation $2 \int f(x)^2 dx$. In \cite{DemTsa} it is proved that martingale solutions of \eqref{SHE} starting from $X_0 =0$ are unique. Our aim is to prove the following result.

\begin{theorem}
\label{t1}
The sequence $\{X_t^n; t \geq 0\}_{n \in \bb N}$ converges in distribution with respect to the uniform topology to the martingale solution of 
\[
\partial_t X = b^2\Delta X + \sqrt{2} \dot{\mc W} 
\]
with initial condition $X_0=0$.
\end{theorem}

As an application of this theorem we will also prove a central limit theorem for the current $J_t^n(0)$:

\begin{theorem}
\label{fBM}
The process $\frac{1}{n^{3/2}} J_t^n(0)$ converges in the sense of finite-dimensional distributions to a fractional Brownian motion of Hurst index $H =\frac{1}{4}$.
\end{theorem}

\section{On additive functionals of Markov processes}
\label{s3}
\subsection{Kipnis-Varadhan inequality}

The main goal of this section is to prove some estimates on the variance of additive functionals of the process $\{\eta_t^n; t \geq 0\}$. In particular we want to prove the analogous of \cite[Prop.~3.7]{GonJar1}, see Proposition \ref{p1}. The proof is almost identical to the proof of \cite[Prop.~3.7]{GonJar1};  we will copy the exposition of \cite[Section 3.2]{GonJar1} and we will explain the differences on the go.

For $f,h \in L^2(\mu^n)$ we write $\<f,h\> = \int fh d\mu^n$. we will omit the dependence in $n$ of $\<\cdot,\cdot\>$, as well as many other quantities. Let $f \in L^2(\mu^n)$ such that $\int f d\mu^n=0$. The $H_{-1}$-norm of $f$ is defined as
\begin{equation}
\label{varh1}
\|f\|_{-1}^2 = \sup_{h} \big\{ 2\<f,h\> - \< h,-L h\>\big\},
\end{equation}
where the supremum runs over local functions in $L^2(\mu^n)$. The importance of the $H_{-1}$-norm is shown by the following inequality:

\begin{proposition}[Kipnis-Varadhan inequality \cite{KipVar,ChaLanOll}]
\label{KV}
For any $T \geq 0$,
\[
\bb E_n \Big[\sup_{0\leq t \leq T} \Big( \int_0^t f(\eta_s^n) ds \Big)^2 \Big] \leq \frac{18 T}{n^4} \|f\|_{-1}^2.
\]
\end{proposition}
This inequality, in the form presented here was proved in \cite{ChaLanOll} following the proof of a slightly different inequality proved in \cite{KipVar}.  

Following \cite{GonJar1}, a very efficient way to estimate the $H_{-1}$-norm of a given function $f$ can be achieved by using the so-called {\em spectral gap inequality}.  In order to state this inequality we need some definitions. For $\ell \in \bb N$ and $x \in \bb N_0$ define $\Lambda_\ell(x) = \{1,\dots,\ell\}$ and $\Omega_\ell(x) = \bb N_0^{\Lambda_\ell(x)}$. For $k \geq 0$ define
\[
\Omega_{k,\ell}(x) = \big\{ \eta \in \Omega_\ell; \sum_{i=1}^\ell \eta(x+i)=k\big\}.
\]
We will write $\Lambda_\ell$, $\Omega_\ell$ and $\Omega_{k,\ell}$ instead of $\Lambda_\ell(0)$, $\Omega_\ell(0)$ and $\Omega_{k,\ell}(0)$ respectively. 
Let $\mu_{k,\ell}$ the uniform measure on $\Omega_{k,\ell}$ and notice that $\mu_{k,\ell}$ is also the restriction of $\mu_\lambda$ to $\Omega_{k,\ell}$. Let us denote by $\<\cdot,\cdot\>_{k,\ell}$ the inner product on $L^2(\mu_{k,\ell})$. For $f \in L^2(\mu^n)$ we denote by $\Var(f)$ the variance of $f$ with respect to $\mu^n$. For $f:\Omega_{k,\ell} \to \bb R$ let $L_\ell f: \Omega_{k,\ell} \to \bb R$ be given by
\[
L_\ell f(\eta) = \sum_{\substack{x, y \in \Lambda_\ell\\ |y-x|=1}} g(\eta(x)) \big(f(\eta^{x,y})-f(\eta)\big).
\] 
We have the following proposition:

\begin{proposition}[Spectral gap inequality \cite{Mor}]
\label{SG}
There exists a universal constant $\kappa_0$ such that 
\[
\<f,f\> \leq \kappa_0(\ell+k)^2 \<f,-L_\ell f\>_{k,\ell}
\]
for any $k, \ell \geq 0$ and any function $f:\Omega_{k,\ell} \to \bb R$ such that $\int f d\mu_{k,\ell}=0$.

\end{proposition}

This proposition was proved in \cite{Mor} for the zero-range process evolving on the complete graph and extended to finite subsets of $\bb Z^d$ using the so called {\em path lemma}. In our one-dimensional situation, a proof can be obtained by coupling with the exclusion process.

For $x \in \bb N_0$ and $\ell \in \bb N$ define $\Lambda_{\ell}(x) = \{x+1,\dots,x+\ell\}$ and
\[
\eta^\ell(x) = \frac{1}{\ell} \sum_{i=1}^\ell \eta(x+i).
\]
Notice that objects like $\mu_{k,\ell}$ or $L_\ell$ can be defined in $\Lambda_\ell(x)$ in a canonical way.
For $f:\Omega \to \bb R$ local such that $\supp(f) \subseteq \Lambda_\ell(x)$ define $\psi_f^\ell: \Omega \to \bb R$ as
$\psi_f^\ell(\eta) = E[f| \eta^\ell(x)]$. Here and in what follows, all conditional expectations are taken with respect to the measure $\mu^n$. Since $\supp(f) \subseteq \Lambda_\ell(x)$, the function $\psi_f^\ell$ does not depend on $n$. We have the following proposition:

\begin{proposition}
Let $f: \Omega \to \bb R$ be a local function such that $\supp(f) \subseteq \Lambda_\ell(x)$ for some $x \in \bb N_0$, $\ell \in \bb N$. Then
\[
\|f-\psi_f^\ell\|_{-1}^2 \leq \kappa_0 \ell^2 \Var\big((1\plus \eta^\ell(x))(f-\psi_f^\ell)\big).
\]
\end{proposition}
This proposition can be proved as Proposition 3.5 in \cite{GonJar1}. We only need to replace the inner products of the form $\<f,h\>$ by
\[
\sum_{k \geq 0} p(k,\ell) \< f,h\>_{k,\ell},
\]
where $p(k,\ell) = \mu^n(\eta^\ell(x) =\frac{k}{\ell})$ and make use of the spectral gap inequality on each subspace $\{\eta^\ell(x) = k\}$. This non uniformity introduces the weighting function $(1+\eta^\ell(x))^2$ into the variance above. Taking these considerations into account, the proof of this proposition can be easily adapted from \cite{GonJar1}, so we omit it.

The next proposition states that functions with supports contained on disjoint intervals are roughly orthogonal with respect to the $H_{-1}$-norm:

\begin{proposition}
\label{port}
Let $m \in \bb N$ be given. Take a sequence $0 \leq x_0 < \dots < x_m \in \bb N_0$ and let $\{f_i; i =1,\dots,m\}$ be a sequence of local functions such that $\supp(f_i) \subseteq \{x_{i-1}+1,\dots,x_i\}$ for any $i$. Define $\ell_i = x_i-x_{i-1}$. Then,
\[
\big\|f_1+\dots+f_m\big\|_{-1}^2 \leq \kappa_0 \sum_{i=1}^m \ell_i^2 \Var\big((1\plus \eta^{\ell_i}(x_{i-1}))(f_i\minus \psi_{f_i}^{\ell_i})\big).
\]
\end{proposition}
Again, the proof of this proposition is a simple adaptation of the proof of Proposition 3.6 in \cite{GonJar1}, so we omit it.
Putting together the estimates of Propositions \ref{KV} and \ref{port} we obtain the following estimate:

\begin{proposition}
\label{p1}
Let $\{f_i; i=1,\dots,m\}$ like in Proposition \ref{port}. Assume in addition that $\psi_{f_i}^{\ell_i}\equiv 0$ for any $i$. Then
\[
\bb E_n \Big[ \sup_{0\leq t \leq T}\Big( \int_0^t \sum_{i=1}^m f_i(\eta_s^n) ds \Big)^2 \Big] 
		\leq \frac{18 \kappa_0 T}{n^4} \sum_{i=1}^m \ell_i^2 \Var((1\plus \eta^{\ell_i}(x_{i-1}))f_i\big)
\]
for any $T \geq 0$.
\end{proposition}

In what follows, Proposition \ref{p1} is all we need from this section. In particular, whenever a estimate like the one stated in Proposition \ref{p1} is available, the methods exposed in the following sections are independent of the spectral gap inequality, and in particular Theorem \ref{t1} holds as soon as Proposition \ref{p1} is available.

\subsection{Integration by parts}

For weighted differences of the functions $g(\eta(x))$, the $H_{-1}$-norm can be estimated without appealing to the spectral gap inequality. In the context of hydrodynamic limits, this estimate is sometimes called the {\em integration by parts formula}. We have the following result:

\begin{proposition}
\label{intpart}
Let $f: \bb N_0 \to \bb R$ be such that $\sum_{x \in \bb N_0} f(x)^2 < +\infty$. Then,
\[
\Big\| \sum_{x \in \bb N_0} \big(g(\eta(x))-g(\eta(x+1))\big) f(x) \Big\|_{-1}^2 \leq \sum_{x \in \bb N_0} f(x)^2.
\]
\end{proposition}
\begin{proof}
Recall the variational formula \eqref{varh1} for the $H_{-1}$-norm. After a change of variables, 
\[
\<g(\eta(x)) -g(\eta(x+1)), h\> = \int g(\eta(x)) \big( h(\eta^{x,x+1})-h(\eta)\big) d\mu^n.
\]
Using the weighted Cauchy-Schwarz inequality and the fact that $g(\eta(x))^2 = g(\eta(x))$ we see that
\begin{equation}
\label{CauSch}
2f(x) \< g(\eta(x)) -g(\eta(x+1)), h\> \leq  \int g(\eta(x)) \big( h(\eta^{x,x+1})-h(\eta)\big)^2 d\mu^n + f(x)^2.
\end{equation}
Notice that
\[
\<h,-Lh\> = \sum_{x\in \bb N_0} \int g(\eta(x)) \big(h(\eta^{x,x+1})-h(\eta)\big)^2d\mu^n.
\]
Summing up \eqref{CauSch} on $x \in \bb N_0$ we obtain the desired bound.
\end{proof}

This estimate is not only simpler than Proposition \ref{p1} but also fundamental in what follows. The point is that the factor $\eta^\ell(x)$ appearing in the general estimates has a very big variance compared to $g(\eta(x))$. We will need the integration by parts formula in order to introduce a spatial average, after which Proposition \ref{p1} starts to be useful.

\section{Proofs}
\label{s4}
In this section we prove Theorems \ref{t1} and \ref{fBM}. The proof follows the classical structure for convergence theorems of stochastic processes: first we prove tightness of the sequence $\{X_t^n; t \geq 0\}_{n \in \bb N}$ with respect to the proper topology. Then we show that any limit point of this sequence is a martingale solution of \eqref{SHE} with zero initial condition. Then we finish the proof of the convergence arguing that by the uniqueness result of \cite{DemTsa}, the sequence $\{X_t^n; t \geq 0\}_{n \in \bb N}$ has a unique limit point. 

\subsection{The martingale decomposition and the continuity relation}
\label{s3.1}
The current processes $\{J_t^n(x); t \geq 0\}$ are Poisson compound processes with disjoint jumps. For any $x \in \bb N$ the processes $\{M_t^n(x); t \geq 0\}$ given by
\begin{equation}
\label{martdec_v0}
M_t^n(x) = J_t^n(x) - n^4\int_0^t \big( g(\eta_s^n(x)) -g(\eta_s^n(x\plus 1))\big) ds
\end{equation}
are martingales of quadratic variation
\[
\<M_t^n(x)\> = n^4 \int_0^t \big( g(\eta_s^n(x))+g(\eta_s^n(x\plus 1))\big) ds.
\]
These formulas also hold for $x =0$ if we use the convention $g(\eta_s^n(0))= \lambda_n$. Since the currents $J_t^n(x)$ have disjoint jumps, the martingales $\{M_t^n(x); t \geq 0\}_{x \in \bb N_0}$ are mutually orthogonal.

In order to simplify the notation, let us write $g_s^n(x) := g(\eta_s^n(x))$. Recall the definition of $X_t^n(f)$ as a sum of currents. We have that
\begin{equation}
\label{martdec}
M_t^n(f) = X_t^n(f) - X_0^n(f)- n^{3/2} \int_0^t \sum_{x \in \bb N_0} \big( g_s^n(x) - g_s^n(x\plus 1)\big) f\big(\tfrac{x}{n}\big) ds
\end{equation}
is a martingale of quadratic variation
\begin{equation}
\label{variation}
\<M_t^n(f)\> = \frac{1}{n} \int_0^t \sum_{x \in \bb N_0} \big( g_s^n(x) + g_s^n(x\plus 1)\big) f\big( \tfrac{x}{n}\big)^2 ds.
\end{equation}
Since $g$ is bounded by $1$ and $\bb E_n[g_s^n(x)] = \lambda_n$, $\<M_t^n(f)\>$ is of order $\mc O(1)$ and it does not vanish in the limit $n \to \infty$. This observation explains the factor $n^{5/2}$ in the definition of $f$: the exponent $5/2$ is tied to the time scale $n^4$. This still does not explain the choice of the time scale $n^4 t$. We will see that $n^4 t$ is the time scale on which the compensator and the martingale part of $X_t^n(f)$ have the same order. We call identity \eqref{martdec} the {\em martingale decomposition} of $X_t^n(f)$. Using the fact that $g_s^n(0) = \lambda_n$ we can rewrite the integral term in \eqref{martdec} as
\begin{equation}
\label{int}
n^{1/2} \int_0^t \sum_{x \in \bb N} \big(g_s^n(x)\minus \lambda_n\big) \nabla_x^n f ds,
\end{equation}
where $\nabla_x^n f := n(f(\frac{x}{n}) - f(\frac{x-1}{n}))$ is a discrete approximation of $f'(\frac{x}{n})$.

Apart from $x=1$, particles are not either created nor destroyed by the dynamics. Therefore, for $x \geq 2$ we have the relation
\begin{equation}
\label{cont}
J_t^n(x\minus 1)-J_t^n(x) = \eta_t^n(x) - \eta_0^n(x).
\end{equation}
In other words, what comes in minus what comes out is equal to what we have minus what we had. For $x=1$ this relation, properly understood, also holds. We call identity \eqref{cont} the {\em continuity relation}. This relation will allow us to express the integral term in \eqref{martdec} in terms of the process $X_t^n$, aside from an error term that goes to $0$ as $n \to \infty$. Let us explain in a heuristic way how will we do this. Notice that $g_s^n(x)$ is a Bernoulli random variable and $\int (1-g(\eta(x))) d\mu^n = 1-\lambda_n$ is asymptotically equivalent\footnote{
Loosely speaking, we will say that two sequences $a_n$, $b_n$ are asymptotically equivalent if $\frac{a_n}{b_n} \to 1$ as $n \to \infty$. To be precise, when the sequences are composed of random variables, we should specify the sense on which the limit holds, but this point will not be of any relevance.}
 to $\frac{1}{\rho_n}$, where $\rho_n = \frac{n}{b} -1= \int \eta(x) d\mu^n$. The so-called {\em Boltzmann-Gibbs principle} states that at the level of fluctuations, any local function can be approximated by a function of the density of particles. By the law of large numbers, this function should be $\frac{1}{\rho}$ in first approximation. Since we are looking at these variables at the level of the central limit theorem, it is reasonable to assume that the average of $g(\eta(x))-\lambda_n$ on the box $\Lambda_\ell(x)$ is well approximated by $\frac{1}{1+\rho_n^2}(\eta^\ell(x) -\rho_n)$.  If we replace $g_s^n(x) - \lambda_n$ by $\frac{1}{1+\rho_n^2}(\eta_s^n(x) -\rho_n)$ in the integral term of \eqref{martdec} and we use the continuity relation \eqref{cont}, we can rewrite this integral term as 
\[
b^2 \int_0^t \frac{1}{n^{5/2}} \sum_{x \in \bb N_0} J_s^n(x) \Delta f \big(\tfrac{x}{n}\big) ds = b^2 \int_0^t X_s^n(\Delta f) ds
\]
plus error terms. It is exactly at this passage that we need to use the hypothesis $f'(0) = 0$. Otherwise a non-vanishing boundary term would appear.
Therefore, assuming the validity of the Boltzmann-Gibbs principle it is possible to rewrite \eqref{martdec} as an approximate closed equation for the current fluctuation field $X_t^n$.

Notice that the function $g_s^n(x)$ is most of the time equal to $1$, being only eventually equal to $0$. Therefore, the dynamics is very singular and the Boltzmann-Gibbs principle can not be proved by the traditional method introduced in \cite{Cha} (see \cite[Chapter 11]{KipLan} for a comprehensive reference). For this reason we need to use the quantitative proof of the Boltzmann-Gibbs principle introduced in \cite{GonJar1}.

\subsection{Tightness}
\label{s3.2}
In this section we prove tightness of the sequence of processes $\{X_t^n; t \geq 0\}_{n \in \bb N}$. As usual, we restrict ourselves to a finite time horizon $[0,T]$. Tightness for the process in $[0,\infty)$ follows pasting intervals of fixed size $T$.

Recall that we are thinking about $\{X_t^n; t \in [0,T]\}$ as a measure-valued process. This is very convenient, since we can reduce tightness considerations to real-valued processes:

\begin{proposition}
\label{realtight}
The family $\{X_t^n; t \in [0,T]\}_{n \in \bb N}$ is tight with respect to the uniform topology if and only if for each function $f \in \mc C_c^\infty([0,\infty))$ the family of real-valued processes $\{X_t^n(f); t \in [0,T]\}_{n \in \bb N}$ is tight.
\end{proposition}
This proposition is standard in the context of hydrodynamic limits of interacting particle systems. A proof of it on finite volume can be found in \cite[Chapter 4]{KipLan}. The proof adapts easily to the unbounded case.

Recall the martingale decomposition \eqref{martdec}. The proof of tightness for $X_t^n(f)$ can be reduced to the proof of tightness of the martingale $M_t^n(f)$, the initial distribution $X_0^n(f)$ and the integral term in \eqref{martdec}. A simple computation shows that the initial distribution $X_0^n(f)$ converges in distribution to a Gaussian random variable of mean $0$ and variance $\frac{1}{b^2} \int F(x)^2 dx$. This integral is finite since $F \in \mc C_c^\infty([0,\infty)$ by definition.

For martingales, powerful methods are available. The integral term will be more demanding, but Kipnis-Varadhan inequality \eqref{KV} coupled with the integration by parts stated in \ref{intpart} will provide the necessary bounds. 

Let us state a convergence criterion for martingales, see \cite[Theorem 2.1]{Whi}
\begin{proposition}
\label{martwhitt}
Let $\{M_t^n; t \in [0,T]\}_{n \in \bb N}$ be a sequence of martingales with $M_0^n \equiv 0$ and let $\Delta_T^n$ the size of the biggest jump of $M_t^n$ in the interval $[0,T]$. Assume that
\begin{itemize}
\item[i)] $\<M_t^n\>$ converges in distribution to $\sigma^2 t$,
\item[ii)] $\Delta_T^n$ converges in probability to $0$. 
\end{itemize}
Then $\{M_t^n; t \in [0,T]\}_{n \in \bb N}$ converges in distribution to a Brownian motion of variance $\sigma^2$.
\end{proposition}

The integral term will be handled with Kolmogorov-Centsov's tightness criterion (see \cite[Exercise 2.4.11]{KarShr}):

\begin{proposition}[Kolmogorov-Centsov's  criterion]
\label{KC}
Let $\{Y_t^n; t \in [0,T]\}_{n \in \bb N}$ be a sequence of real-valued processes with continuous paths. Assume that there exist constants $K,a,a' >0$ such that
\[
E[|Y_t^n-Y_s^n|^a ] \leq K |t-s|^{1+a'}
\]
for any $s,t \in [0,T]$ and any $n \in \bb N$. Then the sequence $\{Y_t^n; t \geq 0\}_{n \in \bb N}$ is tight with respect to the uniform topology.
\end{proposition}

Now we are in position to prove the tightness of $\{X_t^n; t \geq 0\}$. Notice that the current processes $J_t^n(x)$ have jumps of size $1$. Therefore the jumps of $M_t^n(f)$ are at most of size $\frac{\|f\|_\infty}{n^{3/2}}$. In particular the martingales $M_t^n(f)$ satisfy part ii) of the convergence criterion. Recall the martingale decomposition \eqref{martdec} and the formula \eqref{variation} for the quadratic variation of $M_t^n(f)$. In order to prove i), it is enough to observe that
\[
\lim_{n \to \infty} \bb E_n \big[\<M_t^n(f)\>\big] = 2t \|f\|_{L^2(\bb R)}^2
\]
and that
\[
\bb E_n \big[\big( \<M_t^n(f)\>-\bb E_n[\<M_t^n(f)\>]\big)^2\big] \leq \frac{C t^2}{n^3} \sum_{x \in \bb N} f\big(\tfrac{x}{n}\big)^2
\]
for some constant $C$ depending only on $\{\lambda_n\}_{n\in \bb N}$. Therefore, not only the martingale  sequence $\{M_t^n(f); t \in [0,T]\}_{n \in \bb N}$ is tight but it also converges to a Brownian motion of variance $2\int f(x)^2 dx$.

The integral term \eqref{int} is more demanding. Let us introduce the definitions
\[
g^\ell(x) = \frac{g(\eta(x\plus 1))+\dots g(\eta(x\plus \ell))}{\ell}
\]

\[
g_s^{n,\ell}(x) = \frac{ g_s^n(x\plus 1)+\dots + g_s^n(x\plus \ell)}{\ell}.
\]
Let $h: \bb N_0 \to \bb R$ be such that $\sum_{x} h(x)^2 <+\infty$. Notice that $\|\cdot\|_{-1}$ satisfies the triangle inequality. Using the triangle inequality twice we see that 
\begin{align*}
\big\| \sum_{x \in \bb N_0} \big(g(\eta(x)) -g^\ell(x))) h(x) \big\|_{-1} 
		&\leq \frac{1}{\ell} \sum_{j=1}^\ell \sum_{i=1}^j \big\| \sum_{x \in \bb N_0}\big(g(\eta(x\plus i\minus 1))-g(\eta(x\plus i))\big)h(x) \big\|_{-1}\\
		& \leq \ell \Big( \sum_{x \in \bb N_0} h(x)^2\Big)^{1/2}.
\end{align*}
Combining this estimate with Proposition \ref{KV} we obtain the bound
\begin{equation}
\label{est1}
\bb E_n\Big[\Big(n^{1/2}\int_0^t \sum_{x \in \bb N_0} \big(g_s^n(x)-g_s^{n,\ell}(x)\big) \nabla_x^n f ds \Big)^2 \Big]
		\leq \frac{18t \ell^2}{n^3} \sum_{x \in \bb N_0} (\nabla_x^n f)^2,
\end{equation}
which is of order $\mc O(\frac{t\ell^2}{n^2})$. By Cauchy-Schwarz inequality,
\[
\bb E_n\Big[ \Big( n^{1/2} \int_0^t \sum_{x \in \bb N_0} (g_s^{n,\ell}(x) \minus \lambda_n) \nabla_x^n f ds \Big)^2 \Big]
		\leq \frac{b t^2 }{\ell} \sum_{x \in \bb N_0} (\nabla_x^n f)^2,
\]
which is of order $\mc O(\frac{t^2n}{\ell})$. Choosing $\ell = \lceil n t^{1/3}\rceil$ we have just proved that there exists a constant $C:=C(f)$ such that
\[
\bb E_n\Big[\Big( n^{1/2} \int_0^t \sum_{x \in \bb N}\big(g_s^n(x) -\lambda_n\big) \nabla_x^n f ds \Big)^2 \Big]
		\leq Ct^{5/3}
\]
for any $n \in \bb N$ and any $t \in [0,T]$. Since the increments of this process are stationary, we have just proved that the hypothesis of Proposition \ref{KC} holds for the integral term \eqref{int} with $a=2$ and $a' = \frac{2}{3}$. In consequence, the processes
\[
n^{1/2} \int_0^t \sum_{x \in \bb N}\big(g_s^n(x) -\lambda_n\big) \nabla_x^n f ds 
\]
are tight. We conclude that $\{X_t^n(f); t \in [0,T]\}_{n \in \bb N}$ is tight for any $f \in \mc C_c^\infty([0,\infty))$ and by Proposition \ref{realtight} the measure-valued processes $\{X_t^n; t \in [0,T]\}_{n \in \bb N}$ are tight.

\subsection{The Boltzmann-Gibbs principle}

Our objective in this section is to prove the so-called {\em Boltzmann-Gibbs principle} for the integral \eqref{int}:

\begin{proposition}[Boltzmann-Gibbs Principle]
\label{BG}
For any function $f\in \mc C_c^\infty([0,\infty))$,
\begin{equation}\label{BGlim1}
\lim_{n \to \infty} \bb E_n \Big[\Big(n^{1/2}\int_0^t \sum_{x \in \bb N} \big(g_s^n(x) \minus \lambda_n \minus \tfrac{1}{(1 +\rho_n)^2} \big(\eta_s^n(x) \minus \rho_n\big)\big) \nabla_x^n f ds\Big)^2\Big] = 0.
\end{equation}
\end{proposition}

What this proposition is telling us is that the integral \eqref{int} is well approximated as $n \to \infty$ by a linear function of the density of particles. 
The proof of this theorem is somehow winding. We will successively prove that the integral \eqref{int} is asymptotically equivalent to other expressions as $n \to \infty$, until we end up with the density of particles. First we will introduce a spatial average on $g_s^n(x)$. Then the main step comes, which is to replace spatial averages of $g_s^n(x)$ by a function of the particle density 
\[
\eta_s^{n,\ell}(x) = \frac{\eta_s^n(x\plus 1) + \dots + \eta_s^n(x\plus \ell)}{\ell}.
\]
Then we show that this function of $\eta_s^{n,\ell}(x)$ is well approximated by its linearization around $\rho_n$. Finally we undo the spatial average to recover the required estimate.

The following lemma is just a slight modification of estimate \eqref{est1}, so we state it without proof.
\begin{lemma}
\label{l1}
For any function $f \in \mc C_c^\infty([0,\infty))$,
\[
\bb E_n \Big[\Big(n^{1/2}\int_0^t \sum_{x \in \bb N} \big(g_s^n(x) \minus g_s^{n,\ell}(x\minus 1)\big) \nabla_x^n f ds\Big)^2\Big] 
		\leq \frac{18 t \ell^2}{n^3} \sum_{x \in \bb N_0} (\nabla_x^n f)^2.
\]
\end{lemma}

What this lemma has accomplished is to replace $g_s^n(x) -\lambda_n$ by $g_s^{n,\ell}(x\minus 1) -\lambda_n$; the latter has a spatial average of size $\ell$. Notice that the sum $g_s^{n,\ell}(x\minus 1)$ starts at $x$. For $\ell \in \bb N$ define $\psi^\ell:\Omega \to \bb R$ as
\[
\psi^\ell_x(\eta) = E[ g(\eta(x\plus 1))| \eta^\ell(x)].
\]
An explicit computation shows that
\begin{equation}\label{explicit}
\psi^\ell_x(\eta) =1- \frac{1}{1+\frac{\ell}{\ell-1}\eta^\ell(x)}.
\end{equation}
For $x \in \bb N$ and $t \geq 0$, define
\[
\psi_t^{n,\ell} (x) = \psi^\ell_x(\eta_t^n).
\]
The core of the proof is the following lemma:

\begin{lemma}
\label{l2} 
For any function $f \in \mc C_c^\infty([0,\infty))$,
\[
\bb E_n \Big[\Big(n^{1/2}\int_0^t \sum_{x \in \bb N_0} \big(g_s^{n,\ell}(x) \minus \psi_s^{n,\ell}(x)\big) \nabla_{x+1}^n f ds\Big)^2\Big] \leq \frac{C(f) t \ell}{\sqrt{n}},
\]
where $C(f)$ is a constant which depends only on $f$ and the parameters of the model.
\end{lemma}
\begin{proof}
Notice that if $|x-x'| \geq \ell$ then the supports of the functions 
\[
g_s^{n,\ell}(x) \minus \psi_s^{n,\ell}(x), \quad g_s^{n,\ell}(x') \minus \psi_s^{n,\ell}(x')
\]
are disjoint. Therefore, at the price of a multiplicative constant $\ell$ we can put ourselves into the setting of Proposition \ref{p1}.
Therefore, the expectation above is bounded by
\begin{equation}
\label{bound1}
\frac{18\kappa_0 t \ell^3}{n^3} \sum_{x \in \bb N_0} (\nabla_{x+1}^n f)^2 \Var\big(\big(1+\eta^\ell(x)\big)\big(g^\ell(x)-\psi^\ell(x)\big)\big).
\end{equation}

In order to estimate the variance in \eqref{bound1} we use the elementary inequality
\begin{align}\label{var}
\Var (XY) \leq \big[E\big(X-\rho_X\big)^4\big]^{1/2}[EY^4]^{1/2} +2\rho_X[EY^4]^{1/2}\big[E\big(X-\rho_X\big)^2\big]^{1/2} + 
\rho_X^2EY^2,
\end{align}
which comes from the identity 
\[
X^2Y^2 = \big(X-\rho_X\big)^2Y^2 + 2\rho_XY^2\big(X-\rho_X\big) +\rho_X^2Y^2,
\]
valid for any random variables $X$, $Y$  with means $\rho_X$ and zero, respectively. 

Taking $X = 1+\eta^\ell(x)$ and $Y = g^\ell(x)-\psi^\ell(x)$ we see that $EX =  \frac{n}{b}$, $E\big(X-\rho_X\big)^2 \leq \frac{n^2}{b^2\ell}$ \ and $ EY^2 \leq  \frac{b}{n\ell}$.  Furthermore, there exists a finite constant $C$ such that $ E\big(X-\rho_X\big)^4\leq \frac{Cn^4}{\ell^3}$ and $EY^4 \leq \frac{C}{n\ell^3}$.  Putting this estimates into \eqref{var}, we obtain the bound
\[
\Var\big(\big(1+\eta^\ell(x)\big)\big(g^\ell(x)-\psi^\ell(x)\big)\big)\leq
\frac{C n^{3/2}}{\ell^2},
\]
 and we conclude that \eqref{bound1} is bounded by
\[
\frac{Ct \ell}{n^{3/2}} \sum_{x \in \bb N} (\nabla_x^n f\big)^2,
\]
which completes the  proof.
\end{proof}

\begin{proof}[Proof of Proposition {\ref{BG}}]
Let us summarize what we have done up to here. On one hand, combining Lemmas \ref{l1} and \ref{l2} we see that the integral term \eqref{int} is asymptotically equivalent to
\[
n^{1/2} \int_0^t \sum_{x \in \bb N} \big(\psi_s^{n,\ell}(x) \minus \lambda_n\big) \nabla_{x+1}^n f ds
\]
as soon as $\ell \ll n^{1/2}$. On the other hand, it is easy to check that the term $\eta_s^n(x)$ appearing in the formulation of the Boltzmann-Gibbs principle can be replaced by the spatial average $\eta_s^{n,\ell}(x)$ whenever $\ell \ll n$. Therefore, in order to prove Proposition \ref{BG} it remains to verify

\begin{equation}
\label{BGlim2}
\lim_{n \to \infty} \bb E_n \Big[\Big(n^{1/2}\int_0^t \sum_{x \in \bb N} 
\big(\psi_s^{n,\ell}(x) \minus \lambda_n \minus \tfrac{1}{(1+\rho_n)^2} \big(\eta_s^{n,\ell}(x) \minus \rho_n\big)\big) \nabla_{x+1}^{n} f ds \Big)^2\Big] = 0.
\end{equation}

The same argument used at the beginning of the proof of Lemma \ref{l2} permits to bound the preceding expectation by
\begin{equation}\label{e1}
C \ell n^2t^2 \bb E_n \Big[\Big(\psi^{n,\ell}(x) \minus \lambda_n \minus \tfrac{1}{(1+\rho_n)^2} \big(\eta^{n,\ell}(x) \minus \rho_n\big)\Big)^2\Big],
\end{equation}
where $C = C(f)$ is a constant depending solely on $f$.

Observe that the expression squared into the expectation above \emph{almost} corresponds to the error committed in the linearization of the function $h(z)=(1+z)^{-1}$ around $\rho_n$ evaluated at $\eta^{n,\ell}(x)$, which is explicitly given  by
\begin{equation*}
 \frac{1}{1+z}  - \frac{1}{1+\rho_n} + \frac{(z-\rho_n)}{(1+\rho_n)^2} \  = \  \frac{(z-\rho_n)^2}{(1+z) (1+\rho_n)^2} .
\end{equation*}
The word almost in the preceding paragraph is due to the term $\frac{\ell}{\ell - 1}$ appearing in \eqref{explicit}. Rearranging terms in a convenient way, the expectation in \eqref{e1} can be bounded above by twice  
\begin{equation}
\label{e2}
 \frac{1}{ (1+\rho_n)^4}\bb E_n \Big[ \frac{(\eta^{n,\ell}(x)-\rho_n)^4}{(1+\eta^{n,\ell}(x))^2} \Big]
 +
 \frac{1}{ (\ell -1)^2}  \bb E_n \big[ \big(1+ \eta^{n,\ell}(x)\big) ^{-2}  \big].
\end{equation}
Given $a>b$, let us define $ a_n = \tfrac{n}{a} - 1$ Considering separately the cases $\{ \eta^{n,\ell}(x) > a_n\}$ and $\{ \eta^{n,\ell}(x) \leq a_n\}$ we can bound the two expectations above by
\begin{equation}
\label{e3}
\tfrac{a^2}{n^2} \  \bb E_n \Big[ (\eta^{n,\ell}(x)-\rho_n)^4 \Big] +\rho_ n^4 \  \mu^n\big( \eta^{n,\ell}(x) \leq a_n\big) 
  \quad \text{and} \quad
\tfrac{a^2}{n^2} +  \mu^n\big( \eta^{n,\ell}(x) \leq a_n\big),
\end{equation}
respectively.

From \eqref{e1}-\eqref{e3} and the fact that $\bb E_n\big[ (\eta^{ \ell}(x) \minus \rho_n)^4 \big] = \mc O ({ n^4}/{\ell^3})$,  we see that in order to prove \eqref{BGlim2} it is enough to show that
\begin{equation}
\label{ec4.1}
a^2 \big( \tfrac{1}{\ell^2} + \tfrac{\ell}{(\ell -1)^2} \big) + \ell n^2 \big( 1 + \tfrac{1}{(\ell -1)^2} \big)  \mu^n\big( \eta^{n,\ell}(x) \leq a_n\big) 
\end{equation}
goes to zero as $n$ goes to infinity.  According to \eqref{c1} below,  the probability $\mu^n(\eta^{n,\ell}(x) > a_n)$ decays exponentially fast in $\ell$. Therefore, the expression in \eqref{ec4.1} goes to $0$  as soon as $1 \ll \ell$ . This ends the proof of the Boltzmann-Gibbs principle.
\end{proof}

For the sake of completeness, we conclude this subsection with the derivation of  the exponential estimates mentioned above for  tail probabilities of $\eta^\ell(x)$. The estimates are not completely standard due to the increasing density of particles. 

Let $X$ be a random variable with distribution $\Geom(\theta)$: geometric distribution of success probability $\theta$. Notice that $\rho := E[X] = \frac{1 - \theta}{\theta}$. Furthermore, denoting by $\mc M_\rho(\lambda): = E[e^{\lambda X}]$ the moment generating function of $X$, we have that
\[
\mc M_\rho(\lambda) = \frac{1}{1-\rho(e^\lambda-1)}
\]
for $\lambda < \ln(\frac{1+\rho}{\rho})$ and $\mc M_\rho(\lambda) = + \infty$ otherwise. 

Notice that $\eta^{n,\ell}(x)$ is the average of $\ell$ independent geometric random variables with success probability $\frac{b}{n}$, thus $\rho_n := E[\eta^{n}(x)]= \frac{n}{b} - 1 $. Therefore, Cr\'amer's method allows us to obtain the following exponential bounds on tail probabilities: for any $ a \geq 0$
\begin{equation}
\label{c1}
\tfrac{1}{\ell} \log P( \eta^{n,\ell}(x) \leq a) \leq -\mc I_{\rho_n}( a),
\end{equation}
where $\mc I_\rho$ denotes the large deviations rate function associated to geometric distributions of mean $\rho$:
\begin{equation*}
\mc I_\rho(a) := \sup_{\lambda \in \bb R} \big\{ \lambda a - \log \mc M_\rho(\lambda)\big\} = a \log \frac{a(1+\rho)}{\rho(1+a)} - \log \frac{1+a}{1+\rho}.
\end{equation*}

On the other hand, it is no difficult to see that taking $ a_n = \tfrac{n}{a} - 1$ we have
\begin{equation}
\label{c2}
\lim_{n \to \infty} \mc I_{\rho_n}( a_n)=  \tfrac{b}{a} -\log \tfrac{b}{a} -1.
\end{equation}

Observe that the right hand side of the last line coincides with the large deviations rate function associated to exponential distributions. Indeed, this is consistent with the well known fact that if $X^n$ has distribution $\Geom(\tfrac{b}{n})$, then $\tfrac 1n X^n$ converges to an exponential distribution of mean $\frac{1}{b}$.

\subsection{The convergence}

Now we are in place to prove Theorem \ref{t1}. In Section \ref{s3.2} we showed tightness of the sequence $\{X_t^n; t \in [0,T]\}_{n \in \bb N}$ with respect to the uniform topology. Then there are subsequence $n'$ and measure-valued process  $\{X_t^\infty; t \in [0,T]\}$ with continuous paths such that $\{X_t^{n'}; t \in [0,T]\}$ converges to $\{X_t^\infty; t \in [0,T]\}$ in distribution with respect to the uniform topology. Recall the continuity relation stated in \eqref{cont}. We have the relation
\begin{equation}
\label{cont2}
\begin{split}
\frac{b^2}{n^{3/2}} \sum_{x \in \bb N} \big(\eta_s^n(x) -\rho_n\big) \nabla_x^n f
		&= \frac{b^2}{n^{5/2}} \sum_{x \in \bb N_0} J_s^n(x) \Delta_x^n f + \frac{b^2}{n^{3/2}} \sum_{x \in \bb N} \big(\eta_0^n(x) -\rho_n\big) \nabla_x^n f\\
		&-\frac{b^2}{n^{3/2}}  J_s^n(0) \nabla_0^n f,
\end{split}
\end{equation}
where $\Delta_x^n f := n(\nabla_{x+1}^n f - \nabla_x^n f)$ is a discrete approximation of $\Delta f(\frac{x}{n})$. Notice that the last term on the right-hand side of this identity is a boundary term. The following proposition will prove to be useful.

\begin{proposition}
\label{current}
Let $h: \bb N_0 \to \bb R$ be such that $\sum_{x \in \bb N_0} h(x)^2<+\infty$. Then,
\[
\bb E_n \Big[ \Big( \sum_{x \in \bb N_0} J_t^n(x) h(x) \Big)^2 \Big] \leq 32 t n^4 \sum_{x \in \bb N_0} h(x)^2.
\]
\end{proposition}
\begin{proof}
Multiplying \eqref{martdec_v0} by $f(x)$ and adding up on $x$ we obtain the martingale decomposition
\[
\sum_{x \in \bb N_0} J_t^n(x) h(x) = \sum_{x \in \bb N_0} M_t^n(x) h(x) + n^4 \int_0^t \sum_{x \in \bb N_0} \big( g_s^n(x)-g_s^n(x+1)\big) h(x) ds.
\]
Recall that the martingales $M_t^n(x)$ are mutually orthogonal. Therefore,
\[
\bb E_n \Big[ \Big( \sum_{x \in \bb N_0} M_t^n(x) h(x)\Big)^2 \Big] \leq 2 n^4 t \sum_{x \in \bb N_0} h(x)^2.
\]
The integral term can be estimated by combining Propositions \ref{KV} and \ref{intpart}. We obtain the bound
\[
\bb E_n\Big[\Big(n^4 \int_0^t \sum_{x \in \bb N_0} \big( g_s^n(x)-g_s^n(x+1)\big) h(x) ds\Big)^2 \Big]
		\leq 18 t n^4 \sum_{x \in \bb N_0} h(x)^2.
\]
By using the triangle inequality, 
\[
E[(X+Y)^2] \leq (E[X^2]^{1/2} + E[Y^2]^{1/2})^2
\]
the proposition is proved.
\end{proof}

This lemma is useful for two things. First we see that 
\[
\bb E_n \Big[ \Big( \frac{b^2}{ n^{5/2}} \sum_{x \in \bb N_0} J_s^n(x) \big( \Delta_x^n f - \Delta f\big(\tfrac{x}{n}\big)\big) \Big)^2 \Big]
		\leq \frac{32 s b^4}{n} \sum_{x \in \bb N_0} \big( \Delta_x^n f - \Delta f \big( \tfrac{x}{n}\big)\big)^2
\]
and in particular the first sum on the right-hand side of equation \eqref{cont2} is equal to 
\[
b^2 X_s^n(\Delta f) - \frac{b^2}{n^{3/2}} \sum_{x \in \bb N} \big( \eta_0^n(x) -\rho_n\big) f'\big(\tfrac{x}{n}\big)
\]
plus an error term that vanishes in $L^2(\bb P_n)$. The second sum on the right-hand side of \eqref{cont2} is equal to
\[
\frac{b^2}{n^{3/2}} \sum_{x \in \bb N} \big( \eta_0^n(x) -\rho_n\big) f'\big(\tfrac{x}{n}\big)
\]
plus an error term that vanishes in $L^2(\bb P_n)$. Therefore, the two first sums on the right-hand side of \eqref{cont2} are equal to $b^2 X_s^n(\Delta f)$ plus an error term that vanishes in $L^2(\bb P_n)$. It is exactly here that we need the extra sum on the definition of $X_t^n(f)$ introduced in \eqref{def}.

The second use for Proposition \ref{current} is to show that the third term on the right-hand side of \eqref{cont2} is negligible. We have that
\[
\bb E_n\big[ \big(n^{-3/2}J_s^n(0) \nabla_x^n f\big)^2 \big] \leq 32 s n (\nabla_0^n f)^2.
\]
It is exactly at this point that we need to assume that $f'(0)=0$. In that case, $|\nabla_0^n f| \leq \frac{1}{2n} \|\Delta f\|_\infty$ and the boundary term in \eqref{cont2} vanishes in $L^2(\bb P_n)$ as $n \to \infty$. By the Boltzmann-Gibbs principle stated in Proposition \ref{BG}, we conclude that 
\[
n^{1/2}\int_0^t \sum_{x \in \bb N} \big(g_s^n(x) \minus \lambda_n \big) \nabla_x^n f ds = \int_0^t X_s^n(\Delta f) ds
\]
plus an error term that vanishes in $L^2(\bb P_n)$ as $n \to \infty$. In particular, \eqref{martdec} can be written as
\[
M_t^n(f) = X_t^n(f) -X_0^n(f) - \int_0^t X_s^n(b^2\Delta f) ds
\]
plus a term that vanishes in $L^2(\bb P_n)$. Taking limits along the subsequence $n'$ we conclude that for any $f \in \mc C_c^\infty([0,\infty))$ such that $f'(0)=0$, 
\[
X_t^\infty(f) - X_0^\infty(f) - \int_0^t X_s^\infty(b^2 \Delta f) ds
\]
is a Brownian motion of variance $2\int f(x)^2 dx$. Recall that $X_0^n(f)$ converges to a Gaussian random variable of mean $0$ and variance $\frac{1}{b^2}\int f(x)^2 dx$. Therefore $X_0^\infty(f)$ is a spatial white noise of variance $\frac{1}{b^2}$. In other words, $\{X_t^\infty; t \in [0,T]\}$ is a stationary solution of the stochastic heat equation \eqref{SHE} and it is unique in distribution. We conclude that the sequence $\{X_t^n; t \in [0,T]\}_{n \in \bb N}$ has a unique limit point and therefore it converges to it, which proves Theorem \ref{t1}.

\subsection{Proof of Theorem \ref{fBM}}

The proof of Theorem \ref{fBM} follows an idea which we learned from \cite{RosVar}. Let $\varphi:(0,\infty) \to [0,\infty)$ be a smooth function of  support contained on $(0,1)$, such that $\int \varphi(x) dx =1$. For $\epsilon >0$ define $\varphi_\epsilon(x) = \frac{1}{\epsilon} \varphi(\frac{x}{\epsilon})$ and $h_\epsilon(x) = \int_x^\infty \frac{1}{\epsilon} \varphi(\frac{y}{\epsilon})dy$.  Notice that $h_\epsilon(0)=1$, $h_\epsilon(x) =0$ if $x \geq \epsilon$ and $\{\varphi_\epsilon; \epsilon >0\}$ is an approximation of the $\delta$ of Dirac.
Using the continuity relation \eqref{cont} we see that
\[
\frac{1}{n^{3/2}} \Big(J_t^n(0) - \sum_{x=1}^\ell J_t^n(x)\big(h_\epsilon\big(\tfrac{x}{n}\big)-h_\epsilon\big(\tfrac{x-1}{n}\big)\big) \Big) = \frac{1}{n^{3/2}} \sum_{x=1}^\ell \big(\eta_t^n(x)-\eta_0^n(x)\big) h_\epsilon \big(\tfrac{x}{n}\big).
\]
Up to some small error term, the sum on the left-hand side is just $X_t^n(\varphi_\epsilon)$.
In particular, there is a constant $C$ depending only on the parameters of the model and the choice of $\varphi$ such that
\[
\bb E_n\Big[ \Big(\frac{1}{n^{3/2}} J_t^n(0) - X_t^n(\varphi_\epsilon)\Big)^2\Big] \leq C \epsilon.
\]
Taking $n \to \infty$ and then $\epsilon \to 0$ we conclude that
\[
\lim_{n \to \infty} \frac{1}{n^{3/2}} J_t^n(0) = \lim_{\epsilon \to 0} X_t(\varphi_\epsilon)
\]
in the sense of finite-dimensional distributions, as soon as the right-hand side is well defined. As observed in \cite{DemTsa}, this limit exists and it is equal to a fractional Brownian motion of Hurst exponent $H=\frac{1}{4}$, which proves Theorem \ref{fBM}.

\section{Discussion and generalizations}
\label{s5}
\subsection*{Non nearest-neighbor transition rates}
As mentioned in the introduction, there is a combinatorial relation between the zero-range process studied in this article and the exclusion process with symmetric particles with the exception of the leftmost one. The large-density limit considered in this article corresponds to a vanishing density of particles in the exclusion process. Therefore, a proof of the main result of this article using the exclusion process representation and following the steps of \cite{DemTsa} does not seem to be out of reach. Nevertheless, the proof presented here has the advantage of being more general, since it is built upon general properties of interacting particle systems, namely the spectral gap inequality, the equivalent of ensembles and the product structure of the invariant measure. Let us mention here a simple generalization of the model for which our proof can be adapted. Let $p(\cdot)$ be a symmetric transition rate in $\bb Z$ with finite range, but not necessarily equal to 1. For $y \leq 0$, $x >0$ and $\eta \in \Omega$ such that $\eta(x) \geq 1$ we define $\eta^{x,y} = \eta - \delta_x$. The operator
\[
\tilde{L} f( \eta) = \sum_{\substack{x \in \bb N\\y \in \bb Z}} p(y-x) \mathbf{1}\{\eta(x) \geq 1\} \big[ f(\eta^{x,y})-f(\eta)\big] + \lambda \sum_{x \in \bb N} p(x) \big[ f(\eta+\delta_x)-f(\eta)\big] 
\]
defines a zero-range process in $\Omega$ which has the same invariant measure as the zero-range process with nearest-neighbor jumps introduced in Section \ref{s2}. Let us use the same notation $\{\eta_t^n; t \geq 0\}$ for the process generated by $\tilde{L}$. Currents need to be defined in a different way. Let us denote by $R$ the range of $p(\cdot)$. For $x,y \in \bb N$ with $|y-x| \leq R$ we define $J_t^n(x,y)$ the signed number of particles passing from $x$ to $y$ and let us define $J_t^n(0,x)$ as the number of particles created at $x$ minus the number of particles destroyed at $x$. Let us define
\[
J_t^n(x) = \sum_{\substack{0 \leq z \leq x \\ y >x}} J_t^n(z,y).
\]
The current fluctuation field $X_t^n$ is defined like in Section \ref{s2.2} using this version of the current processes. Then Theorem \ref{t1} holds as well:

\begin{theorem}
The sequence $\{X_t^n; t \geq 0\}_{n \in \bb N}$ converges in distribution with respect to the uniform topology to the martingale solution of 
\[
\partial_t X = b^2\sigma^2 \Delta X + \sqrt{2} \dot{\mc W} 
\]
with initial condition $X_0=0$, where $\sigma^2 = \sum_{z>0} z^2 p(z)$.
\end{theorem}

The interested reader will not find any difficulty to adapt the proof of Theorem \ref{t1} for this case. The only relevant difference is the following version of the continuity equation \eqref{cont}: for any $x<y$ such that $y-x>R$,
\[
J_t^n(x) - J_t^n(y) = \sum_{z=x+1}^y \big(\eta_t^n(z) -\eta_0^n(z)\big). 
\]

\subsection*{General interaction rates}

As mentioned before, the main properties of the zero-range process used in the proof of Theorem \ref{t1} are a sharp bound for the largest eigenvalue of the dynamics restricted to a finte box, the explicit knowledge of the invariant measure of the system and the equivalence of ensembles. With these three properties in hand, the proof follows through. In \cite{Nag}, the spectral gap inequality for the zero-range process with interaction rate $g(k) = k^\gamma$, $\gamma \in (0,1)$ was derived, as well as the estimates needed to obtain the equivalence of ensembles. Therefore, the proof of Theorem \ref{t1} could in principle be adapted to the zero-range process with interaction rate $g(k) = k^\gamma$.  The limiting equation will be the same of Theorem \ref{t1}, with diffusion coefficient and noise variance depending on the parameters of the model.

\section*{Acklowledgements}

F.H.~thanks FAPERJ for its support through the grant E-26/112.076/2013.
M.J.~thanks CNPq for its support through the grant 401628/2012-4 and FAPERJ for its support through the grant JCNE E17/2012. M.J.~would like to thank Amir Dembo for valuable discussions. F.V.~would like to thank the warm hospitality of IMPA, where this work was done.

%
%
%
%

%
%
%
%
%
%
%
%
%
%
%
%
%
%
%
%
%
%

\end{document}